\documentclass[final,hidelinks]{siamart1116Arxiv}

\usepackage{tikz,pgfplots}
\pgfplotsset{compat=newest}
\usetikzlibrary{decorations.pathreplacing,calc}

\usepackage{xstring} % for IfEqCase

\usepackage[shortlabels]{enumitem}

\usepackage{cite}

\usepackage{amsfonts}

\usepackage{mathtools}
\usepackage{autonum} % showonlyrefs option of mathtools incompatible with cleveref package

\numberwithin{theorem}{section}
\newsiamremark{rmrk}{Remark}

\newcommand{\Div}{{\rm div}\,}
\newcommand{\st}{{\ | \ }} % "such that" in the definition of a set
\newcommand{\iii}{{\vert\kern-0.25ex\vert\kern-0.25ex\vert}}
\newcommand{\sol}{{u}}
\newcommand{\test}{{v}}
\newcommand{\gtest}{{\chi}}
\newcommand{\normal}{{\boldsymbol \nu}}
\newcommand{\rd}[1]{%
  \IfEqCase{#1}{%
    {x}{\xi}%
    {y}{\eta}%
    {(x,y)}{(\xi,\eta)}%
  }[\PackageError{\rd}{Opzione non definita per rd: #1}{}]%
}
\newcommand{\allhold}{{\Omega^*}}
\newcommand{\angolo}{{\psi}}
\newcommand{\curv}{{\mathcal{H}_\omega}}
\newcommand{\avW}[2]{{\tilde{W}^{#1}_{#2}(0,1)}}
\newcommand{\Wg}{\avW{1}{\infty}}
\newcommand{\Wc}{\avW{1}{1}}
\newcommand{\zW}[2]{%
  \mathrel{\vbox{\offinterlineskip\ialign{%
    \hfil##\hfil\cr
    $\scriptscriptstyle\circ$\cr
    \noalign{\kern0.1ex}
    ${\,W^{#1}_{#2}}$\cr
}}}(\Omega^\omega)}
\newcommand{\zWo}[2]{%
  \mathrel{\vbox{\offinterlineskip\ialign{%
    \hfil##\hfil\cr
    $\scriptscriptstyle\circ$\cr
    \noalign{\kern0.1ex}
    ${\,W^{#1}_{#2}}$\cr
}}}(\Omega^0)}
\newcommand{\zWsmall}[3]{%
  \mathrel{\vbox{\offinterlineskip\ialign{%
    \hfil##\hfil\cr
    $\scriptscriptstyle\circ$\cr
    \noalign{\kern0.1ex}
    $\scriptstyle{\,W^{#1}_{#2}}$\cr
}}}(\Omega^{#3})}
\newcommand{\Wsol}{\zW{1}{p}}
\newcommand{\Wtest}{\zW{1}{q}}
\newcommand{\Wtestsmall}{\zWsmall{1}{q}{\omega}}
\newcommand{\Wsolo}{\zWo{1}{p}}
\newcommand{\Wtesto}{\zWo{1}{q}}
\newcommand{\Wtestosmall}{\zWsmall{1}{q}{0}}
\newcommand{\Wp}{{W^1_p(\Omega^\omega)}}
\newcommand{\Wq}{{W^1_q(\Omega^\omega)}}
\newcommand{\Wpo}{{W^1_p(\Omega^0)}}
\newcommand{\Wqo}{{W^1_q(\Omega^0)}}
\newcommand{\Winf}{{W^1_\infty(0,1)}}
\newcommand{\Wone}{{W^1_1(\Omega^\omega)}}
\newcommand{\Vo}{%
  \mathrel{\vbox{\offinterlineskip\ialign{%
    \hfil##\hfil\cr
    $\scriptscriptstyle\circ$\cr
    \noalign{\kern0.4ex}
    ${\,V_h}$\cr
}}}}
\newcommand{\So}{{\tilde{S}_h}}
\newcommand{\interpol}{{\mathcal{I}}}
\newcommand{\sgn}{{\rm sgn}}
\newcommand{\epsw}{{\varepsilon_{fb}}}

\renewcommand{\hat}[1]{\widehat{#1}}

\newcommand{\TheTitle}{A free-boundary problem with moving contact points}
\newcommand{\TheAuthors}{I. Fumagalli}
\headers{\TheTitle}{\TheAuthors}
\title{{\TheTitle}\thanks{\funding{This work was supported by Moxoff s.p.a. (\url{www.moxoff.com})}}}
\author{
  Ivan Fumagalli\thanks{MOX - Department of Mathematics, Politecnico di Milano, piazza Leonardo da Vinci 32, 20133 Milano, Italy
    (\email{ivan.fumagalli@polimi.it}).}
}

\ifpdf
\hypersetup{
  pdftitle={\TheTitle},
  pdfauthor={\TheAuthors}
}
\fi

\begin{document}

\maketitle

\begin{abstract}
 This paper concerns the theoretical and numerical analysis of a free boundary problem for the Laplace equation, with a curvature condition on the free boundary.
 This boundary is described as the graph of a function, and contact angles are imposed at the moving contact points.
 The equations are set in the framework of classical Sobolev Banach spaces, and existence and uniqueness of the solution are proved via a fixed-point iteration, exploiting a suitably defined lifting operator from the free boundary.
 The free-boundary function and the bulk solution are approximated by piecewise linear finite elements, and the well-posedness and convergence of the discrete problem are proved.
 This proof hinges upon a stability result for the Riesz projection onto the discrete space, which is separately proven and has an interest per se.
\end{abstract}



%%%%%%%%%%%%%%%%%%%%%%%%%%%%%%%%%%%%%%%%%%%%%%%%%%%%%%%%%%%%%%%%%%%%%%%%%%%%%%%%
%%%%%%%%%%%%%%%%%%%%%%%%%%%%%%%%%%%%%%%%%%%%%%%%%%%%%%%%%%%%%%%%%%%%%%%%%%%%%%%%
\section*{Introduction}
\label{sec:intro}
\addcontentsline{toc}{section}{\nameref{sec:intro}}
%%%%%%%%%%%%%%%%%%%%%%%%%%%%%%%%%%%%%%%%%%%%%%%%%%%%%%%%%%%%%%%%%%%%%%%%%%%%%%%%
%%%%%%%%%%%%%%%%%%%%%%%%%%%%%%%%%%%%%%%%%%%%%%%%%%%%%%%%%%%%%%%%%%%%%%%%%%%%%%%%

Free boundary problems governed by PDEs present many different features, that make their theoretical and numerical analysis a challenging task.
In the present work, a free boundary problem for the Laplacian with a curvature condition is considered, in the presence of moving contact points.
The free boundary is described as the graph of a function, and Neumann conditions are imposed at the end points, in order to account for the enforcement of a contact angle.

A milestone work on this subject is represented by \cite{SS91}.
In that paper, a free boundary problem for the fully Dirichlet Laplacian was investigated, in the case of {\em fixed} contact points.
The well-posedness of the continuous problem, and the stability and convergence of its piecewise linear finite element approximation were proved.
Few extensions of that work are available in the literature, in the direction of generalizing the results to the Stokes operator \cite{GiraultNochettoScott}, potential flows \cite{BaiChooChungKim} or optimal control problems governed by free boundary systems \cite{ANS14}.
In the case of shape optimization problems, in which moving boundary are similarly entailed, different techniques have been employed, to draw a theoretical and numerical analysis of the problem (see, e.g.~\cite{FPVShOpt,KinigerVexler,Eppler}).
However, the presence of moving contact points is still an open problem, in the theoretical literature.
Indeed, as stated in the conclusions of \cite{SS91}, this objective is not straightforwardly achievable, and a careful consideration of the boundary conditions is crucial.

This paper aims at extending the results of \cite{SS91} to the case of a free boundary with moving contact points.
This represents a first step towards a better theoretical and numerical description of free surface flows with moving contact lines, which are relevant in many applications and whose study is the subject of an active computational literature (see, e.g.~\cite{MovingCL,Gerbeau,Walker,Scardovelli}).
The free-boundary problem is set in the framework of classical $W^k_p$ Sobolev spaces, and in order to prove its well-posedness, a proper definition of a lifting operator is introduced, connecting the bulk problem with the equation governing the free boundary.
The continuous problem is, then, discretized by means of a piecewise linear finite element method, and the stability and convergence of the resulting scheme are proved, resorting to the proof of a $W^1_p$ stability result for the Riesz projection onto the discrete space.
In this regard, a result presented in \cite{RannacherScott} for a fully Dirichlet bulk problem is extended to the case of mixed boundary conditions.

The present paper is made of two parts.
\cref{sec:pb} is devoted to the definition of the free boundary problem under inspection and to the analysis of its weak formulation. The proof of its well-posedness via a fixed point iteration is provided in \cref{subsec:wp}.
In \cref{sec:discrete}, a piecewise linear finite element approximation is introduced for both the bulk solution and the free-boundary function.
Stability and convergence of the numerical scheme is stated, hinging upon the stability of the Riesz projection onto the discrete scheme, to whose proof \cref{subsec:Riesz} is dedicated.

%%%%%%%%%%%%%%%%%%%%%%%%%%%%%%%%%%%%%%%%%%%%%%%%%%%%%%%%%%%%%%%%%%%%%%%%%%%%%%%%
%%%%%%%%%%%%%%%%%%%%%%%%%%%%%%%%%%%%%%%%%%%%%%%%%%%%%%%%%%%%%%%%%%%%%%%%%%%%%%%%
\section{Problem definition}\label{sec:pb}
%%%%%%%%%%%%%%%%%%%%%%%%%%%%%%%%%%%%%%%%%%%%%%%%%%%%%%%%%%%%%%%%%%%%%%%%%%%%%%%%
%%%%%%%%%%%%%%%%%%%%%%%%%%%%%%%%%%%%%%%%%%%%%%%%%%%%%%%%%%%%%%%%%%%%%%%%%%%%%%%%

\begin{figure}
 \centering
 \begin{tikzpicture}[scale=1]
   \begin{scope}[local bounding box=scope1]
		\draw[thick] (0,3) -- node[left]{$\Sigma^0$} (0,0)-- node[below]{$\Sigma_b$} (3,0) -- node[right]{$\Sigma^0$} (3,3) -- node[above]{$\Gamma^0$} (0,3);
		\node[anchor=north] at (1.5,1.5) {$\Omega^0$};
    \node[fill,draw,circle, minimum width=2pt, inner sep=0pt, label={[label distance=1pt, inner sep=0pt]45:$\rd{(x,y)}$}] at (1,2) {};
    \coordinate (riferimento) at (3,0);
    \coordinate (arrLeft) at (4,1.5);
   \end{scope}
   \begin{scope}[shift={($(riferimento)+(5,0)$)}]
  	\draw[thick] (0,3) -- node[left]{$\Sigma^\omega$} (0,0)-- node[below]{$\Sigma_b$} (3,0) -- node[right]{$\Sigma^\omega$} (3,2.7);
	 	\draw[thick] (0,3) to[out=0,in=150] (1.5,3.2) node[above]{$\Gamma^\omega$} to[out=-30,in=-160] (3,2.7);
	  \node[anchor=north] at (1.5,1.5) {$\Omega^\omega$};
    \node[fill,draw,circle, minimum width=2pt, inner sep=0pt, label={[label distance=1pt, inner sep=0pt]45:$(x,y)$}] at (1,2) {};
    \draw (3,2.4) arc (-90:-160:0.3) node[below,pos=0.7]{$\theta$};
    \coordinate (arrRight) at (-1,1.5);
	\end{scope}
  \draw[thick, ->] (arrLeft) to[out=30,in=150] node[above]{$\Psi^\omega$} (arrRight);
 \end{tikzpicture}
 \caption{Reference domain (left) and actual configuration (right) for the problem.\label{fig:domain}}
\end{figure}
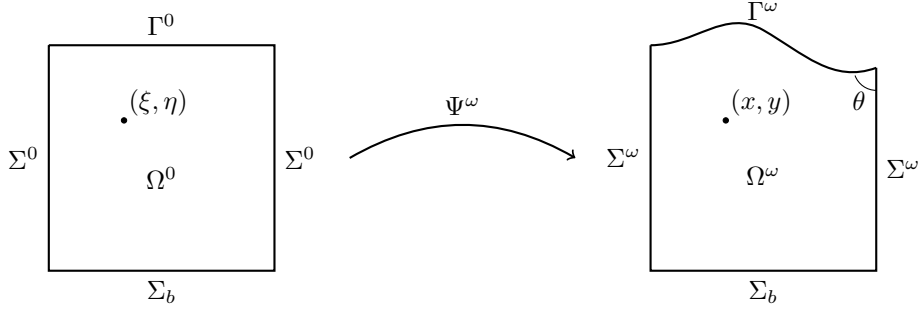

Let $\Omega^\omega\in\mathbb{R}^2$ be a free-boundary, bounded domain defined as
\begin{equation}\label{eq:domain}
    \Omega^\omega = \{(x,y) \st x\in (0,1), y\in (0,1+\omega(x))\},
\end{equation}
where $\omega\in W^1_\infty(0,1)$ is a function such that $\|\omega\|_{W^1_\infty}<1$.
We denote by $\Gamma^\omega$ the top boundary of $\Omega^\omega$:
\begin{equation}
    \Gamma^\omega = \{(x,1+\omega(x)) \st x\in(0,1)\}.
\end{equation}
As displayed in \cref{fig:domain}, the lateral boundary of the domain is named $\Sigma^\omega$, whereas $\Sigma_b$ is the bottom side.
This domain $\Omega^\omega$ is the image of the unit square $\Omega^0=(0,1)^2$ through the $W^1_\infty$-regular map
\begin{equation}
    \Psi^\omega \colon \Omega^0 \to \mathbb{R}^2, \qquad (x,y) = \Psi^\omega(\rd{x},\rd{y}) = \left(\rd{x} , (1+\omega(\rd{x}))\rd{y}\right).
\end{equation}
One can notice that, being $\omega$ bounded, all the possible $\Omega^\omega$ are contained in the all-holding domain $\allhold=(0,1)\times(0,2)$.

Given a Lebesgue space exponent $p\in[2,\infty)$, with its conjugate $q\colon 1/p+1/q=1$, the free-surface problem addressed in the present work is to find $(\omega, \sol)\in W^1_\infty(0,1)\times W^1_p(\Omega^\omega)$ such that
\begin{equation}\label{eq:pb}
    \begin{cases}
	-\Delta \sol = 0 & \text{in } \Omega^\omega, \\
	\sol = g & \text{on } \Sigma_b \cup \Gamma^\omega, \\
	\partial_\normal \sol = 0 & \text{on } \Sigma^\omega, \\
	\partial_\normal \sol = \gamma\curv & \text{on } \Gamma^\omega, \\
	\omega'(0)=0,\qquad \omega'(1) = \angolo, \\
	\int_0^1 \omega(t)\,dt=0,
    \end{cases}
\end{equation}
where the function $g\in W^1_\infty(\allhold)$ is given, $\normal$ is the unit outward normal vector of the domain, $\curv=-\left(\omega'/\sqrt{1+(\omega')^2)}\right)'$ is the curvature of the top boundary, $\angolo=\cot\theta$ is a prescribed steepness of the top boundary at its right end, and $\gamma>0$ represents a surface tension coefficient.
The conditions on the first derivative of $\omega$ prescribe the angles between the free boundary $\Gamma^\omega$ and the wall $\Sigma^\omega$, that have to be $\pi/2$ at the left contact point and $\theta$ at the right one.
In particular, the left condition $\omega'(0)=0$ is the one that arises if the line $x=0$ is a symmetry axis, and we look at $\Omega^\omega$ as the section of a planarly symmetric or axisymmetric domain: indeed, this kind of symmetries are often involved in the applications (see, for example, \cite{axisymm,MovingCL,Yamamoto}).
\begin{rmrk}
    The last equation in problem \cref{eq:pb} is a zero-average constraint on the function $\omega$.
    This is necessary to ensure the uniqueness of $\omega$, since this function appears in the equations only through its derivatives.
    This constraint corresponds to an area/volume constraint on the domain.
\end{rmrk}
\begin{rmrk}
    Throughout the present work, the linearized curvature\linebreak[4] $\curv:=-\omega''/\sqrt{1+(\omega')^2}$ will be considered.
    This choice prevents the functional setting of the problem from getting technically over-complicated, without affecting the generality of the results, as pointed out also in \cite{ANS14, SS91}.
    For simplicity, we use the same symbol $\curv$ already adopted for the complete curvature introduced above.
\end{rmrk}

%%%%%%%%%%%%%%%%%%%%%%%%%%%%%%%%%%%%%%%%%%%%%%%%%%%%%%%%%%%%%%%%%%%%%%%%%%%%%%%%
\subsection{Weak formulation of the problem}\label{subsec:weakpb}
%%%%%%%%%%%%%%%%%%%%%%%%%%%%%%%%%%%%%%%%%%%%%%%%%%%%%%%%%%%%%%%%%%%%%%%%%%%%%%%%

As stated above, the variational framework in which the problem at hand is set involves the classical Sobolev spaces $W^1_\infty(0,1)$, for the free-boundary function $\omega$, and $W^1_p(\Omega^\omega)$, for the bulk solution $\sol$.
In order to account for the boundary conditions and the zero-average constraint, the following spaces are introduced:
\begin{align}
    \avW{k}{s} &= \left\{\omega\in W^k_s(0,1) \ \left\vert\ \int_0^1\omega\, dt=0\right.\right\}, \\
    \zW{k}{s} &= \{\sol\in W^k_s(\Omega^\omega) \st \sol=0 \text{ on } \Sigma_b\cup\Gamma^\omega\}, \\
    \mathbb{W} &= \Wg\times\Wsol, \\
    \mathbb{Z} &= \Wc\times\Wtest,
\end{align}
where $q=p/(p-1)$.
\begin{rmrk}[Poincar\'e inequality]\label{rmrk:poincare}
    In both $\avW{1}{s}$ and $\zW{1}{s}$ Poincar\'e inequality holds, for any $s\in[1,\infty]$ (see, e.g., \cite[Theorems 8.11-8.12]{poincareMedia}).
    For each $s\in[1,\infty]$ we will denote by $c_s,C_s$ the positive constants such that
    \begin{gather}
     \|\omega\|_{W^1_s(0,1)}\leq c_s\|\omega'\|_{L^s(0,1)},\quad\forall\omega\in  W^1_s(0,1), \\
     \|\sol\|_{W^1_s(\Omega^\omega)}\leq C_s\|\nabla\sol\|_{L^s(\Omega^\omega)},\quad\forall\sol\in W^1_s(\Omega^\omega),
    \end{gather}
    with $c_s,C_s$ independent of $\omega$, thanks to the assumption $\|\omega\|_{W^1_\infty(0,1)}<1$.
\end{rmrk}

Problem \cref{eq:pb} can be stated in weak form as:
Find $(\omega,\sol-g)\in\mathbb{W}$ such that, for any $(\gtest,\test)\in\mathbb{Z}$,
\begin{equation}\label{eq:weak}
    \begin{cases}
	a^\omega(\sol,\test) = 0, \\
	b(\omega,\gtest) = a^\omega(\sol,E^\omega\gtest) + \angolo\gtest(1),
    \end{cases}
\end{equation}
where
\begin{align}
    a^\omega(\sol,\test) &= \int_{\Omega^\omega} \nabla\sol\cdot\nabla\test\,dx, \\
    b(\omega,\gtest) &= \int_0^1 \omega'\gtest'\, dt,
\end{align}
and $E^\omega\gtest$ is a suitable extension of $\gtest$ onto $\Omega^\omega$, that is going to be defined in Lemma \ref{lmm:extension}.
Indeed, provided that such an extension is zero on $\Sigma_b$, we can write that
\begin{align}
    \int_0^1 \partial_\normal\sol(t,1+\omega(t))&\gtest(t)\sqrt{1+(\omega'(t))^2}\,dt = \int_{\Gamma^\omega}\partial_\normal\sol\,\gtest\,d\Gamma \\= &\int_{\Gamma^\omega}\partial_\normal\sol\, E^\omega\gtest \, d\Gamma = \int_{\Omega^\omega} \nabla\sol\cdot\nabla E^\omega\gtest\,dx,
\end{align}
and since the boundary conditions on $\Gamma$ require
\begin{align}
 b(\omega,\gtest) &= \int_0^1\omega'\gtest' = -\int_0^1\omega''\gtest + \angolo\gtest(1) \\
 &= \int_0^1\partial_\normal\sol(t,1+\omega(t))\gtest(t)\sqrt{1+(\omega'(t))^2}+\angolo\gtest(1),
\end{align}
we have that \cref{eq:weak} is actually the weak formulation of \cref{eq:pb}.

\begin{lemma}[Extension]\label{lmm:extension}
    For every $\gtest\in W^1_1(0,1)$ there exists an extension $E^\omega\gtest\in W^1_q(\Omega^\omega)$, as long as $q<2$, such that $E^\omega\gtest|_{\Gamma^\omega}=\gtest$, $E^\omega\gtest|_{\Sigma_b}=0$, and
    \begin{equation}\label{eq:extension}
	\|E^\omega\gtest\|_\Wq \leq c_0(\|\omega\|_\Winf)\|\gtest\|_\Wone,
    \end{equation}
    where $c_0$ depends only on $\|w\|_{W^1_\infty(0,1)}$, and not on the extension.
\end{lemma}
\begin{proof}
    Given some $\gtest\in W^1_1(0,1)$, let $\overline{\gtest}:\partial\Omega^0\to\mathbb{R}$ be an extension of $\gtest$ to the whole boundary of the reference domain $\Omega^0$, such that $\overline{\gtest}|_{\Gamma^0}=\gtest$, $\overline{\gtest}|_{\Sigma_b}=0$, and $\overline{\gtest}(t,\rd y)=\rd{y}\gtest(t)$, $t=0,1$.
    Thanks to the compact embedding $W^1_1(0,1)\subset W^{1-1/q}_q(0,1)$, holding for $q<2$, $\gtest$ is $W^{1-1/q}_q$-regular, and so is $\overline{\gtest}$ \cite{ASV88}.
    Therefore, $\overline{\gtest}$ can be extended as a function $\hat E\gtest: \Omega^0\to\mathbb{R}$.
    Thanks to the theory of traces, this bulk extension can be done in such a way that $\hat E\gtest \in W^1_q(\Omega^0)$ and
    \begin{equation}
	\|\hat E\gtest\|_{W^1_q(\Omega^0)} \leq C\|\overline\gtest\|_{W^{1-1/q}_q(\partial\Omega^0)} \leq \hat c_0\|\gtest\|_{W^1_1(0,1)},
    \end{equation}
    with $\hat c_0$ independent of $\gtest,\omega$.
    Eventually, $\hat E\gtest$ can be continuously mapped to a $W^1_q$-regular $E^\omega\gtest\colon\Omega^\omega\to\mathbb R$ by means of the change of variables induced by $\Psi^\omega$, and the following steps conclude the proof:
  \begin{equation}
    \|E^\omega\gtest\|_\Wq \leq c(\|\omega\|_{\Winf})\|\hat E\gtest\|_{W^1_q(\Omega^0)} \leq \hat c_0c(\|\omega\|_\Winf)\|\gtest\|_{W^1_1(0,1)}.
  \end{equation}
  \
\end{proof}

\begin{rmrk}
    The extension $E^\omega\gtest$ is not unique, but this does not affect problem \cref{eq:weak}, since for any given pair of admissible extensions $E_1^\omega,E_2^\omega$, we have that\linebreak[4]
    $E_1^\omega\gtest-E_2^\omega\gtest\in\Wtest$ for any $\gtest\in \Wone$, whence $a^w(\sol,E_1^\omega\gtest-E_2^\omega\gtest)=0$.
\end{rmrk}

%%%%%%%%%%%%%%%%%%%%%%%%%%%%%%%%%%%%%%%%%%%%%%%%%%%%%%%%%%%%%%%%%%%%%%%%%%%%%%%%
\subsection{Well-posedness of the problem}\label{subsec:wp}
%%%%%%%%%%%%%%%%%%%%%%%%%%%%%%%%%%%%%%%%%%%%%%%%%%%%%%%%%%%%%%%%%%%%%%%%%%%%%%%%

In this section, the proof of the well-posedness of the weak problem \cref{eq:weak} is addressed.
Following the ideas of \cite{SS91,ANS14}, the well-posedness of the individual problems on $\omega$ and $\sol$ is going to be proved, and then, the result for the coupled problem will be achieved via a fixed-point iteration.
The fixed-point iteration that will be considered is the following:
given $(\omega,\sol)\in\mathbb W$, let $\widetilde\omega\in\Wg$ be the solution of
\begin{equation}\label{eq:gammatilde}
    b(\widetilde\omega,\gtest) = a^\omega(\sol,E^\omega\gtest)+\angolo\gtest(1), \qquad \forall\gtest\in\Wc,
\end{equation}
and then let $\widetilde\sol-g\in\Wsol$ solve
\begin{equation}\label{eq:soltilde}
    a^{\widetilde\omega}(\widetilde\sol,\test) = 0, \qquad \forall\test\in W^1_q(\Omega^{\widetilde\omega}).
\end{equation}
We are going to show that this is actually a fixed-point iteration in the compact set
\begin{equation}\label{eq:fpset}
    \mathbb B = \left\{(\omega,\sol)\in \mathbb W \st \|\omega\|_\Winf\leq\epsw, \|\sol\|_\Wp\leq \varepsilon\right\},
\end{equation}
for a suitable choice of $0<\epsw,\varepsilon<1$, and that the map
\begin{equation}\label{eq:fpmap}
    T : \mathbb B \to \mathbb W, \qquad T(\omega,\sol) = (T_1(\omega,\sol), T_2(T_1(\omega,\sol),\sol)) = (\widetilde\omega,\widetilde\sol),
\end{equation}
is a contraction map.
To this aim, it is worth to introduce some notation related to the mapping $\Psi^\omega$ induced by $\omega$.
We denote by $\hat\cdot$ the composition with $\Psi^\omega$: if not clear from the context, it will be explicitly stated which particular choice for $\omega$ is considered.
With this notation, we introduce the bilinear form
\begin{equation}
 \hat a(\cdot,\cdot;\omega) : \Wpo\times\Wqo\to\mathbb R \quad \text{ such that } \quad \hat a(\hat\sol,\hat\test;\omega) = \int_{\Omega^0}\nabla\hat\sol^TA^\omega\nabla\hat\test,
\end{equation}
where $A^\omega=|\det\nabla\Psi^\omega|(\nabla\Psi^\omega)^{-1}(\nabla\Psi^\omega)^{-T}$.
We point out that $\hat a(\hat\sol,\hat\test;\omega) = a^\omega(\sol,\test)$ for any $\sol\in\Wsol$ and $\test\in\Wtest$.
The properties of the forms $\hat a$ and $a^\omega$ are very strictly related, thanks to the following \cref{lmm:normequiv} on the equivalence of norms, which is based on the inequality
 \begin{equation}\label{eq:CA}
  \mathbf v^TA^\omega\mathbf v\leq C_A |\mathbf v|^2, \qquad \forall\mathbf v\in\mathbb R^2,
 \end{equation}
  holding for $\omega$ ranging in the unit ball of $W^1_\infty(0,1)$ and being $C_A>0$ independent of $\omega$.
\begin{lemma}\label{lmm:normequiv}
 There exists a constant $c_n>0$ such that
 \begin{equation}
  \frac{1}{c_n}\|\sol\|_{W^1_p(\Omega^\omega)} \leq \|\hat\sol\|_{W^1_p(\Omega^0)} \leq c_n\|\sol\|_{W^1_p(\Omega^\omega)},
 \end{equation}
 for any $\sol\in W^1_p(\Omega^w)$, $p\in[1,\infty]$ and $\omega\in W^1_\infty(0,1)$ such that $\|\omega\|_{W^1_\infty(0,1)}< 1$.
\end{lemma}

The well-posedness of the problems \cref{eq:gammatilde,eq:soltilde} hinges upon the results of continuity and inf-sup stability of the forms $a^\omega,b,\hat a$ collected in the following statement.

\begin{lemma}\label{lmm:continfsup}
 There exists a constant $\hat\alpha>0$ such that the following inequalities hold for any $\omega\in\Wg$, $\gtest\in\Wc, \hat\sol\in\Wsolo, \hat\test\in\Wtesto, \sol\in\Wsol, \test\in\Wtest$:
 \begin{align}
  b(\omega,\gtest)&\leq|\omega|_{W^1_\infty}|\gtest|_{W^1_1}\leq\|\omega\|_{W^1_\infty}\|\gtest\|_{W^1_1},
    \label{eq:bcont}\phantom{\ref{eq:bcont}}\\
  \|\omega\|_{W^1_\infty}&\leq c_\infty^2\sup_{\gtest\in\Wc\setminus{0}}\frac{b(\omega,\gtest)}{\|\gtest\|_{W^1_1}},
    \label{eq:binfsup}\phantom{\ref{eq:binfsup}} \\
  \hat a(\hat\sol,\hat\test;\omega)&\leq C_A\|\hat\sol\|_{W^1_p}\|\hat\test\|_{W^1_q},
    \label{eq:hatacont}\phantom{\ref{eq:hatacont}}\\
  \|\hat\sol\|_{W^1_p}&\leq\hat\alpha\sup_{\hat\test\in\Wtestosmall\setminus{0}}\frac{\hat a(\hat\sol,\hat\test;\omega)}{\|\hat\test\|_{W^1_q}},
    \label{eq:hatainfsup}\phantom{\ref{eq:hatainfsup}}\\
  a^\omega(\sol,\test)&\leq|\sol|_{W^1_p}|\test|_{W^1_q} \leq \|\sol\|_{W^1_p}\|\test\|_{W^1_q},
    \label{eq:acont}\phantom{\ref{eq:acont}}\\
  \|\sol\|_{W^1_p}&\leq\hat\alpha c_n^2\sup_{\test\in\Wtestsmall\setminus{0}}\frac{a^\omega(\sol,\test)}{\|\test\|_{W^1_q}}.
    \label{eq:ainfsup}\phantom{\ref{eq:ainfsup}}
 \end{align}
\end{lemma}
\begin{proof}
 Starting from the proof of \cref{eq:binfsup}, let us take a fixed $\omega\in\Wg$.
 Recalling \cref{rmrk:poincare} and noticing that $L^\infty(0,1)=(L^1(0,1))$, we have that
 \begin{equation}
  \|\omega\|_{W^1_\infty(0,1)} \leq c_\infty |\omega|_{W^1_\infty(0,1)} = c_\infty \sup_{f\in L^1(0,1)} \frac{\int_0^1\omega'f\,dt}{\|f\|_{L^1(0,1)}}.
 \end{equation}
 Now, since for any $f\in L^1(0,1)$ we can define a function $\gtest(x)=\int_0^x f\,dt-\int_0^1f\,dt$ such that $\chi\in\Wc$ and $\chi'=f$, we can write
 \begin{equation}\label{eq:infsupb}\begin{aligned}
  \|\omega\|_{W^1_\infty(0,1)} &\leq c_\infty \sup_{f\in L^1(0,1)} \frac{\int_0^1\omega'f\,dt}{\|f\|_{L^1(0,1)}}\leq c_\infty \sup_{\gtest\in\Wc}\frac{\int_0^1\omega'\gtest'\,dt}{|\gtest|_\Wone} \\
  &\leq c_\infty^2 \sup_{\gtest\in\Wc}\frac{b(\omega,\gtest)}{\|\gtest\|_\Wone},
 \end{aligned}\end{equation}
 that is exactly \cref{eq:binfsup}.
 
 Concerning the form $\hat a$, we observe that a possible expression for the constant $C_A$ defined in \cref{eq:CA} is
 \begin{equation}
  C_A = \max_{\|\omega\|_{W^1_\infty}}\|A^\omega\|_{L^\infty} = \max_{\|\omega\|_{W^1_\infty}} \max\left\{1+\|\omega\|_{L^\infty} ; \left\|\frac{1+(\omega')^2}{1+\omega}\right\|_{L^\infty}\right\}.
 \end{equation}
 With this definition, we can notice that any eigenvalue $\lambda$ of $A^\omega$ fulfills
 \begin{equation}
  \frac{1}{2C_A}\leq\lambda\leq 2C_A,
 \end{equation}
 and hence the results in \cite{Meyers} yield the existence of a suitable $\hat\alpha$ and the validity of \cref{eq:hatacont,eq:hatainfsup}.

 The proof concludes by noticing that the remaining inequalities can be proven by means of Cauchy-Schwarz inequality, \cref{lmm:normequiv} and the inequalities \cref{eq:binfsup,eq:hatacont,eq:hatainfsup} just demonstrated.
\end{proof}

Now, we can prove the well-posedness of the individual problems \cref{eq:gammatilde} and \cref{eq:soltilde}, that can be stated as in the following result.
\begin{proposition}\label{prp:T1T2}
    The solution maps $T_1:\mathbb B\to\Wg$ and $T_2:\mathbb B\to\Wsol$ defined in \cref{eq:fpmap} are injective and continuous.
\end{proposition}
\begin{proof}
  Employing the continuity and inf-sup inequalities for $a^\omega(\cdot,\cdot)$ and $b(\cdot,\cdot)$, stated in \cref{lmm:continfsup}, we can prove the existence and uniqueness of the solutions to problems \cref{eq:gammatilde} and \cref{eq:soltilde}, that is equivalent to the thesis.
  
  Starting with problem \cref{eq:gammatilde}, uniqueness comes directly from \cref{eq:binfsup}, whereas for existence some more steps are needed.
  Let $\varphi$ be the linear functional over $\Wc$ defined by the right-hand side of \cref{eq:gammatilde}, namely $\varphi(\gtest) = a^\omega(\sol, E^\omega\gtest) + \angolo\gtest(1)$.
  Being $a^\omega(\cdot,\cdot)$ continuous, $\varphi\in(\Wc)'\subset(\avW{1}{2})'$ and hence Riesz theorem implies the existence of a $\widetilde\omega\in\avW{1}{2}$ such that $b(\widetilde\omega,\gtest) = \varphi(\gtest)$ for any $\gtest\in\avW{1}{2}$.
  Now, it is enough to show that $\widetilde\omega$ actually belongs to $\Wg$ and it is the solution of problem \cref{eq:gammatilde}.
  We employ a density argument, like in \cite{ANS14}.
  Given a Cauchy sequence $\{\gtest_n\}_{n\in\mathbb{N}}$ in $\avW{1}{2}$, such a sequence is Cauchy also w.r.t.\ the full norm of $\Wc$, due to the continuous embedding $\avW{1}{2}\hookrightarrow\Wc$.
  Therefore, thanks to the continuity of $b(\widetilde\omega,\cdot)$ and $\varphi(\cdot)$, $\widetilde\omega$ fulfills \cref{eq:gammatilde} for a test function $\gtest$ given by the $\Wc$-limit of $\gtest_n$.
  Finally, being $W_2^1(0,1)$ dense in $W_1^1(0,1)$, a sequence $\{\gtest_n\}$ can be constructed for any $\gtest\in\Wc$, yielding that $\widetilde\omega$ is indeed the solution of \cref{eq:gammatilde}.
  The bound on $\|\widetilde\omega\|_{W^1_\infty(0,1)}$ required to state that $\widetilde\omega\in W^1_\infty(0,1)$ derives directly from the inf-sup inequality \cref{eq:binfsup}.
  
  Regarding problem \cref{eq:soltilde}, since $\Wsol,\Wtest$ are reflexive spaces, uniqueness comes from the application of Brezzi-Ne\v{c}as-Babu\v{s}ka theorem (see, e.g., \cite[Theorem~2.6]{ErnGuermond}), together with the inf-sup stability \cref{eq:ainfsup} of the form $a^\omega$.
  
  Eventually, the continuity of the maps $T_1,T_2$ stems from that of the forms $a^{\widetilde\omega}(\cdot,\cdot)$ and $b(\cdot,\cdot)$.
\end{proof}

We are now ready to state the main result for the existence of the solution to \cref{eq:weak}.
\begin{theorem}\label{th:contrazione}
 Let
 \begin{equation}\label{eq:hatfpset}
    \hat{\mathbb B} = \left\{(\omega,\hat\sol)\in \Wg\times\Wsolo \st \|\omega\|_\Winf\leq\epsw, \|\hat\sol\|_\Wpo\leq \varepsilon\right\}.
\end{equation}
Then, there exist $\overline{\angolo},\delta>0$ and $P>2$ such that, if $|\angolo|<\overline{\angolo}$, $\mathbf\|g\|_{W^1_p(\allhold)}<\delta$ for some $p\in(2,P)$, and $g\in W^2_s(\allhold)$ for some $s>2$, the map $\hat T:\hat{\mathbb B}\to\hat{\mathbb B}$ defined as $\hat T(\omega,\hat\test)=(T_1(\omega,\test),\hat{T_2(\omega,\test)})$ is a contraction w.r.t.\ the norm
 \begin{equation}
  \iii(\omega,\hat\sol)\iii = \varepsilon\|\omega\|_{\Winf} + \epsw\|\hat\sol\|_{W^1_p(\Omega^0)},
 \end{equation}
 for $\varepsilon$ and $\epsw$ sufficiently small.
\end{theorem}
\begin{proof}
 At first, we are going to show that the image of $\hat{\mathbb B}$ through the map $\hat T$ is indeed contained in $\hat{\mathbb B}$.
 Let $\widetilde\omega=T_1(\omega,\sol), \widetilde\sol=T_2(\widetilde\omega,\sol)$.
 From \cref{eq:binfsup}, the expression of problem \cref{eq:gammatilde} and the continuity \cref{eq:hatacont} of the form $\hat a$, we have that
 \begin{align}
  \|\widetilde\omega\|_{W^1_\infty(0,1)} &\leq c_\infty^2 \sup_{\chi\in\Wc\setminus\{0\}}\frac{b(\omega,\chi)}{\|\chi\|_{W^1_1(0,1)}} = c_\infty^2\sup_{\chi\in\Wc\setminus\{0\}}\frac{\hat a(\hat\sol,\hat E\chi;\omega) + \angolo\chi(1)}{\|\chi\|_{W^1_1(0,1)}} \\
  &\leq C_A\hat c_0\|\hat\sol\|_{W^1_p(\Omega^0)}+|\angolo| \leq C_A\hat c_0\varepsilon+\overline{\angolo},
 \end{align}
 whence $(\widetilde\omega,\hat\sol)\in\hat{\mathbb B}$ if $\varepsilon,\overline{\angolo}$ are chosen in such a way that $C_Ac_0\varepsilon+\overline{\angolo}<\epsw$.
 Analogous arguments yield
 \begin{align}
  \|\hat{\widetilde\sol}\|_{W^1_p(\Omega^0)} &\leq \|\hat g\|_{W^1_p(\Omega^0)} + \hat\alpha \sup_{\hat\test\in\Wtestosmall\setminus\{0\}}\frac{\hat a(\hat{\widetilde\sol}-\hat g,\hat\test;\widetilde\omega)}{\|\hat\test\|_{W^1_q(\Omega^0)}} \\
  &= \|\hat g\|_{W^1_p(\Omega^0)} + \hat\alpha \sup_{\hat\test\in\Wtestosmall\setminus\{0\}}\frac{-\hat a(\hat g,\hat\test;\widetilde\omega)}{\|\hat\test\|_{W^1_q(\Omega^0)}} \leq (1+\hat\alpha C_A)\|\hat g\|_{W^1_p(\Omega^0)},
 \end{align}
and hence the final solution $\hat T(\omega,\hat\sol)\in\hat{\mathbb B}$, as long as $\delta<(1+\hat\alpha C_A)^{-1}\varepsilon$.

 Now, in order to show that $\hat T$ is a contraction map, we introduce $\widetilde\omega_i=T_1(\omega_i,\sol_i)$ and $\widetilde\sol_i=T_2(\widetilde\omega_i,\sol_i)$, where $(\omega_i,\sol_i), i=1,2,$ are given elements of $\mathbb B$.
 Following the proof of \cite[Theorem 2.1]{SS91}, one can show that
 \begin{equation}\label{eq:deltagamma}\begin{aligned}
  \|\widetilde\omega_1-\widetilde\omega_2\|_{W^1_\infty(0,1)} &\leq c_\infty^2\hat c_0\max\left\{C_A;c\frac{1+\epsw}{1-\epsw}\right\}\\
  &\qquad\cdot\left(\varepsilon\|\omega_1-\omega_2\|_{W^1_\infty(0,1)}+\epsw\|\hat\sol_1-\hat\sol_2\|_{W^1_p(\Omega^0)}\right)\\
  &= c_\infty^2\hat c_0\max\left\{C_A;c\frac{1+\epsw}{1-\epsw}\right\}\iii(\omega_1-\omega_2, \hat\sol_1-\hat\sol_2)\iii,
 \end{aligned}\end{equation}
 where $\hat\sol_i$ is the preimage of $\sol_i$ via the map $\Psi^{\omega_i}:\Omega^0\to\Omega^{\omega_i}$.
 In order to control $\|\hat{\widetilde\sol_1}-\hat{\widetilde\sol_2}\|_{W^1_p(\Omega^0)}$, instead, some more steps are due: indeed, the difference $\hat{\widetilde\sol_1}-\hat{\widetilde\sol_2}$ does not belong to $\Wsol$, since it is equal to $\hat g_1-\hat g_2$ on $\Gamma^0$, where $\hat g_1,\hat g_2$ are different preimages of the Dirichlet datum $g$ via the maps induced by $\widetilde\omega_1,\widetilde\omega_2$, respectively.
 Employing the triangle inequality and the inf-sup condition \cref{eq:hatainfsup} of the form $\hat a(\cdot,\cdot;\widetilde\omega_1)$ gives
 \begin{equation}\label{eq:deltauhat1}\begin{aligned}
  \|\hat{\widetilde\sol_1}-\hat{\widetilde\sol_2}\|_{W^1_p(\Omega^0)}&\leq \|\hat g_1-\hat g_2\|_{W^1_p(\Omega^0)} \\ &\qquad+\hat\alpha\sup_{\hat\test\in\Wtestosmall\setminus\{0\}}\frac{\hat a\left(\hat{\widetilde\sol_1}-\hat g_1 - \hat{\widetilde\sol_2}+\hat g_2,\hat\test;\widetilde\omega_1\right)}{\|\hat\test\|_{W^1_q(\Omega^0)}}.
 \end{aligned}\end{equation}
 Noticing that
 \begin{equation}
  \hat a(\hat{\widetilde\sol_1},\hat\test;\widetilde\omega_1) = \hat a(\hat{\widetilde\sol_2},\hat\test;\widetilde\omega_2) = 0\qquad \forall\hat\test\in\Wtesto,
 \end{equation}
 we can bound the second term of \cref{eq:deltauhat1} as follows:
 \begin{align}
  \hat a\left(\hat{\widetilde\sol_1}-\hat g_1 - \hat{\widetilde\sol_2}+\hat g_2,\hat\test;\widetilde\omega_1\right) &= \hat a(\hat g_2-\hat g_1,\hat\test;\widetilde\omega_1) - \hat a(\hat{\widetilde\sol_2},\hat\test;\widetilde\omega_1) \\
  &=\hat a(\hat g_2-\hat g_1,\hat\test;\widetilde\omega_1) + \hat a(\hat{\widetilde\sol_2},\hat\test;\widetilde\omega_2) - \hat a(\hat{\widetilde\sol_2},\hat\test;\widetilde\omega_1) \\
  &\leq C_A\|\hat g_1-\hat g_2\|_{W^1_p(\Omega^0)}\|\hat\test\|_{W^1_q(\Omega^0)} \\
  &\qquad+ \frac{1+\epsw}{1-\epsw}\varepsilon\|\widetilde\omega_1-\widetilde\omega_2\|_{\Winf}\|\hat\test\|_{W^1_p(\Omega^0)}.
 \end{align}
 Thanks to the assumption $g\in W^2_s(\allhold)$, the difference between the two preimages of this function can be controlled in terms of the difference in the maps:
 \begin{equation}
  \|\hat g_1-\hat g_2\|_{W^1_p(\Omega^0)}\leq C_g\|g\|_{W^2_s(\allhold)}\|\widetilde\omega_1-\widetilde\omega_2\|_{\Winf}^{1-2/s} \leq C_g\|g\|_{W^2_s(\allhold)}\|\widetilde\omega_1-\widetilde\omega_2\|_{\Winf}.
 \end{equation}
 Therefore, we can conclude that
 \begin{equation}\label{eq:deltauhat}
  \|\hat{\widetilde\sol_1}-\hat{\widetilde\sol_2}\|_{W^1_p(\Omega^0)}\leq \left[(1+\hat\alpha C_A)C_g\|g\|_{W^2_s(\allhold)}+\hat\alpha\frac{1+\epsw}{1-\epsw}\varepsilon\right]\|\widetilde\omega_1-\widetilde\omega_2\|_{\Winf}.
 \end{equation}
 Eventually, merging \cref{eq:deltagamma} and \cref{eq:deltauhat} yields the thesis, provided that
 \begin{gather}
  \|g\|_{W_s^2(\allhold)} < \delta < (1+\hat\alpha C_A)^{-1}\varepsilon, \label{eq:dummy1}\phantom{\ref{eq:dummy1}}\\
  (1+\hat\alpha C_A)C_g\epsw\delta + \hat\alpha\varepsilon\epsw\frac{1+\epsw}{1-\epsw} + c_\infty^2\hat c_0\max\left\{C_A;c\frac{1+\epsw}{1-\epsw}\right\}\varepsilon < 1,\label{eq:dummy2}\phantom{\ref{eq:dummy2}}\\
  \overline\angolo < \epsw - C_A\hat c_0\varepsilon.\label{eq:psibar}
 \end{gather}

\end{proof}

\begin{rmrk}
  The last part of the proof of \cref{th:contrazione} requires, among other bounds, a restriction on the admissible steepness $\angolo$.
  In particular, an interpretation of inequality \cref{eq:psibar} is that the limitation on the angle comes from a trade-off between the bound $\varepsilon$ on the bulk solution and the bound $\epsw$ on the free-boundary function.
  Appropriately balancing this trade-off, we can obtain different bounds on $\angolo$, any of which entails $\overline\angolo<1$.
  Anyway, this latter limitation is not much restrictive, since it allows $\angolo$ to range approximately in $(65^\circ, 115^\circ)$: many fluid dynamics applications actually involve contact angles that lie in this range \cite{MovingCL, Yamamoto}.
\end{rmrk}

Thanks to the equivalence of norms stated in \cref{lmm:normequiv}, the following result is a direct consequence of \cref{th:contrazione}.

\begin{corollary}\label{cor:wellposedness}
 If $\|g\|_{W^2_s(\allhold)}$ and $\angolo$ are sufficiently small, then for any given $p\in(2,P)$ problem \cref{eq:weak} admits a unique solution $(\omega,\sol)\in\mathbb B$, that can be obtained by fixed point iterations, starting with any initial guess $(\omega^{(0)},\sol^{(0)})\in\mathbb B$.
\end{corollary}

\begin{rmrk}
 The statement of \cref{cor:wellposedness}, as well as all the previous results, still hold if a non-homogeneous bulk equation,
 \begin{equation}
  -\Delta\sol = f\qquad\text{ in }\Omega,
 \end{equation}
 is considered, provided that $\|f\|_{(\Wtestsmall)'}$ is sufficiently small.
\end{rmrk}

%%%%%%%%%%%%%%%%%%%%%%%%%%%%%%%%%%%%%%%%%%%%%%%%%%%%%%%%%%%%%%%%%%%%%%%%%%%%%%%%
%%%%%%%%%%%%%%%%%%%%%%%%%%%%%%%%%%%%%%%%%%%%%%%%%%%%%%%%%%%%%%%%%%%%%%%%%%%%%%%%
\section{The discrete problem}
\label{sec:discrete}
%%%%%%%%%%%%%%%%%%%%%%%%%%%%%%%%%%%%%%%%%%%%%%%%%%%%%%%%%%%%%%%%%%%%%%%%%%%%%%%%
%%%%%%%%%%%%%%%%%%%%%%%%%%%%%%%%%%%%%%%%%%%%%%%%%%%%%%%%%%%%%%%%%%%%%%%%%%%%%%%%

Let us introduce a triangulation $\mathcal T_h^0$ for the domain $\Omega^0$, with a discretization step $h$, and denote by $\{\mathbf n_k=(\xi_k,\eta_k)\}_{k=1}^{N_h}$ the nodes of this mesh, with the first $N_\Gamma+1$ nodes lying on $\Gamma^0$ and ordered from left to right.
On $\mathcal T_h^0$, we set up a conforming finite element space
\begin{equation}
 \Vo_{,0} = \{\test_h\in C^0(\overline{\Omega^0}) \st \test_h|_K\in\mathbb P_1(K)\ \forall K\in\mathcal T_h^0, \text{ and }\test_h|_{\Gamma^0\cup\Sigma_b}=0\}
\end{equation}
 of piecewise linear functions with zero trace on $\Gamma^0\cup\Sigma_b$.
Considering the first coordinate of the points of the mesh $\mathcal T_h$ lying on $\Gamma^0$, we denote by $\mathcal S_h=\{[\xi_k,\xi_{k+1}]\}_{k=1}^{N_\Gamma}$  the corresponding one-dimensional grid for the interval $[0,1]$.
On this second mesh, we introduce the finite element space $\So$ of zero-mean piecewise linear functions:
\begin{equation}
 \So = \left\{\gtest_h\in C^0([0,1]) \left\vert \gtest_h|_{[\xi_k,\xi_{k+1}]}\in\mathbb P_1([\xi_k.\xi_{k+1}])\ \forall k=1,\dots,N_\Gamma, \text{ and} \int_0^1\gtest_h=0\right.\right\}.
\end{equation}

Given an element $\omega_h$ of this space, the domain $\Omega^0$ can be transformed into a domain $\Omega^{\omega_h}$ via a piecewise linear map $\Psi_h^{\omega_h}$, and the space $\Vo_{,0}$ is mapped to an other piecewise-linear finite element space $\Vo$ on the new domain.
In these settings, the classical finite element formulation for problem \cref{eq:weak} reads as follows:\\
Find $(\omega_h,\sol_h-g_h)\in\So\times\Vo$ such that
\begin{equation}\label{eq:FE}\left\{\begin{aligned}
 a^{\omega_h}(\sol_h,\test_h) &= 0&&\forall\test_h\in\Vo,\\
 b(\omega_h,\gtest_h) &= a^{\omega_h}(\sol_h,E_h^{\omega_h}\gtest_h) + \angolo\gtest_h(1) &&\forall\gtest_h\in\So,
\end{aligned}\right.\end{equation}
where $g_h$ is the piecewise linear interpolation of the Dirichlet datum $g$.

As for the continuous problem, the discrete problem \cref{eq:FE} requires a proper definition of a lifting operator $E_h^{\omega_h}:\So\to\Vo$.
For the problem at hand, we can simply define it as
\begin{equation}
 E_h^{\omega_h}\chi_h = (J_h\hat E\gtest_h) \circ \Psi_h^{\omega_h},
\end{equation}
where $J_h: H^1(\Omega^0)\to\Vo$ is the classical Cl\'ement interpolator \cite{QV}.
It is worth remarking that, differently from \cite{SS91}, one can not consider a discrete extension $E_h\gtest_h$ having support on the only upper side $\Gamma^{\omega_h}$, because it would spoil the nullity of the difference $\hat E\gtest_h-J_h\hat E\gtest_h$ on the lateral boundary $\Sigma^0$: this subject will be better discussed in \cref{rmrk:Eo}.

In order to prove the well-posedness of problem \cref{eq:FE}, as well as the stability and convergence properties of the approximation, we need to show that the forms $a^{\omega_h}$ and $b$ are inf-sup stable also in the discrete spaces, and that the functional $\gtest_h\mapsto a^{\omega_h}(\sol_h,E_h^{\omega_h}\gtest_h)$ is continuous.
To this aim, two main conditions are required:

\begin{enumerate}
 \item\label{it:EWo} $\hat E\chi_h-J_h\hat E\chi_h\in\Wtesto$, where $\hat E$ is defined as in the proof of \cref{lmm:extension}, that is, this difference is an admissible test function for the continuous problem on the domain $\Omega^0$;
 \item\label{it:Riesz} the Riesz projection operator $R_h:\zWo{1}{2}\to\Vo_{,0}$, defined as the solution operator of problem
\begin{equation}\label{eq:Riesz}
 \int_{\Omega^0}\nabla R_h\sol\cdot\nabla\test_h = \int_{\Omega^0}\nabla\sol\cdot\nabla\test_h,\qquad\forall\test_h\in \Vo_{,0},
\end{equation}
is stable in $W^1_p(\Omega^0)$ for any $p\in[1,\infty)$, namely
\begin{equation}\label{eq:Rieszstable}
 \exists C_R>0\quad\text{such that}\quad\|R_h\sol\|_{W^1_p(\Omega^0)}\leq C_R\|\sol\|_{W^1_p(\Omega^0)},\quad\forall\sol\in\Wsolo.
\end{equation}
\end{enumerate}

To prove condition \ref{it:EWo}, we observe that any discrete test function $\gtest_h\in\So$ belongs to $W^1_1(0,1)$, thus the proof of \cref{lmm:extension} can be followed.
Therefore, $\hat E\gtest_h\in\Wtesto$ and, since also piecewise polynomials belong to $W^1_q(\Omega^0)$, the difference $\hat E\gtest_h-J_h\hat E\gtest_h$ is in $\zWo{1}{q}$.
Concerning the second condition, some more work is needed, in order to deal with mixed boundary conditions: this discussion is postponed to \cref{subsec:Riesz}.

Under the above conditions, the proofs of \cite[Proposition 3.3]{SS91} and of all the consequent results therein can be followed without any modifications: in those results, the role of having fully Dirichlet boundary conditions is to provide Poincar\'e inequality and the stability of the Riesz projection, both of which still hold for our spaces $\zWo{1}{p},\zW{1}{p},\Vo$.
Thus, we can state the following collective result:

\clearpage
\begin{theorem}
\label{th:discr}
 \ 
 \begin{enumerate}[(i)]
  \item\label{it:discrcontrazione} Under the hypotheses of \cref{th:contrazione}, the discrete problem \cref{eq:FE} admits a unique solution $(\omega_h,\sol_h)$ in
  \begin{equation}
   \mathbb B_h = \mathbb B\cap(\Vo\times\So),
  \end{equation}
  which can be computed by fixed point iterations like in the continuous case, starting from any $(\omega_h^0,\sol_h^0)\in\mathbb B_h$.
  
\item\label{it:convergence} If $\varepsilon$ and $\epsw$ are sufficiently small, and the solution $(\omega,\sol)\in\mathbb B$ of the continuous problem belongs to $W^2_\infty(0,1)\times W^2_p(\Omega^\omega)$ for some $p>2$, then there are two constants $C,h_0\in(0,\infty)$ such that, for any $h\in(0,h_0]$,
  \begin{equation}
   \|\omega-\omega_h\|_{W^1_\infty(0,1)}+\|\sol\circ\Psi^\omega - \sol_h\circ\Psi_h^{\omega_h}\|_{W^1_p(\Omega^0)} \leq Ch(\|\omega\|_{W^2_\infty(0,1)}+\|\sol\|_{W^2_p(\Omega^\omega)}).
  \end{equation}
 \end{enumerate}
\end{theorem}

\begin{rmrk}\label{rmrk:Eo}
 As observed in the conclusions of \cite{SS91}, the proof of a convergence result like \ref{it:convergence} of \cref{th:discr} exploits that the difference $e_h = \hat E\gtest_h-J_h\hat E\gtest_h$ belongs to $\{\test\in W^1_q(\Omega^0) \st \test=0 \text{ on }\partial\Omega\}$.
 This is straightforwardly granted in the fully-Dirichlet case with fixed contact points considered in \cite{SS91}, since $e_h|_{\Gamma^0}=0$ by definition, and the restrictions of both $\hat E\gtest_h$ and $J_h\hat E\gtest_h$ to $\Sigma^0\cup\Sigma_b$ are set to zero.
 In the present work, instead, the desired property holds because $\hat E\gtest_h$ is linear on the Neumann boundary $\Sigma^0$, and the interpolator $J_h$ preserves linear functions.
\end{rmrk}

%%%%%%%%%%%%%%%%%%%%%%%%%%%%%%%%%%%%%%%%%%%%%%%%%%%%%%%%%%%%%%%%%%%%%%%%%%%%%%%%
\subsection{Stability of Riesz projection}\label{subsec:Riesz}
%%%%%%%%%%%%%%%%%%%%%%%%%%%%%%%%%%%%%%%%%%%%%%%%%%%%%%%%%%%%%%%%%%%%%%%%%%%%%%%%

The present section is devoted to the proof of the inequality \cref{eq:Rieszstable} for the Riesz projection operator defined in \cref{eq:Riesz}.
Since this result may have an interest per se, we collect here the geometrical settings in which our proof takes place:
\begin{itemize}
 \item we consider a rectangular domain $\Omega$ (like the square $\Omega^0$ of the previous sections);
 \item we denote by $\Gamma_D$ a couple of opposite boundary sides of $\Omega$ (that corresponds to $\Gamma^0\cup\Sigma$, in the previous sections);
 \item in the different problems that will be introduced, homogeneous Dirichlet boundary conditions will be enforced, on the boundary $\Gamma_D$, whilst homogeneous Neumann boundary conditions will be applied elsewhere.
\end{itemize}
In particular, the last point ensures some compatibility conditions that provide second-order Sobolev regularity of the functions involved, thanks to results like those in \cite{Lorenzi}.

In order to tackle the main result of the present section, we have to extend the following technical result by Rannacher and Scott:

\begin{lemma}[{\cite[section 3]{RannacherScott}}]\label{lmm:actualRSlemma}
 Denoting by $H^1_0(\Omega)$ the space
 \begin{equation}
  H^1_0(\Omega) = \{v\in W^1_2(\Omega) \st v|_{\partial\Omega}=0\},
 \end{equation}
 let functions $f\in H^1_0(\Omega)$ and $\mathbf f\in[H^1_0(\Omega)]^2$ be given, and let $\check w\in H^1_0(\Omega)$ be such that
 \begin{equation}\begin{cases}
  -\Delta \check{w} = f+\Div\mathbf f &\qquad\text{ in }\Omega, \\
  \check{w} = 0 &\qquad\text{ on }\partial\Omega.
 \end{cases}\end{equation}
 Then, for any convex polygonal domain $\Omega$, there exists an $\alpha_\Omega\in(0,1]$ such that for all parameter values $\alpha\in(0,\alpha_\Omega]$ the following a priori estimates hold 
 \begin{enumerate}[(i)]
  \item if $f\equiv0$,
    \begin{equation}
     \int_\Omega\sigma_{\mathbf z,\zeta}^{2+\alpha}|\nabla^2 \check w|^2 \leq c\left(\int_\Omega\sigma_{\mathbf z,\zeta}^{2+\alpha}|\Div\mathbf f|^2 + \alpha^{-1}\zeta^{-2}\int_\Omega\sigma_{\mathbf z,\zeta}^{2+\alpha}|\mathbf f|^2\right);
    \end{equation} \label{it:fo}
  \item if $\mathbf f\equiv\mathbf 0$,
    \begin{equation}
     \int_\Omega\sigma_{\mathbf z,\zeta}^{-2-\alpha}|\nabla^2 \check w|^2 \leq c\alpha^{-1}\zeta^{-2}\int_\Omega\sigma_{\mathbf z,\zeta}^{2-\alpha}|\nabla f|^2;
    \end{equation} \label{it:fvo}
 \end{enumerate}
 where $\nabla^2$ denotes the Hessian matrix, and $\sigma_{\mathbf z,\zeta}:\Omega\to[0,\infty)$ is defined in terms of an arbitrary point $\mathbf z\in\Omega$ and an arbitrary scalar $\zeta\in\mathbb R$, as $\sigma_{\mathbf z,\zeta}(\mathbf x) = \sqrt{|\mathbf x-\mathbf z|^2 + \zeta^2}
 $.
\end{lemma}

In particular, we need to consider mixed boundary conditions, instead of fully Dirichlet ones, and thus to prove the following result:

\begin{lemma}\label{lmm:RSpar3}
 Let $w$ be the solution of the following problem over a rectangle $\Omega$:
 \begin{equation}\label{eq:pbRSpar3}\begin{cases}
  -\Delta w = f+\Div\mathbf f & \text{ in }\Omega,\\
  w = 0 & \text{ on }\Gamma_D,\\
  \partial_\normal w = 0 & \text{ on }\partial\Omega\setminus\Gamma_D,
 \end{cases}\end{equation}
 where $\Gamma_D$ is the union of a pair of opposite sides of $\Omega$, and $f\in H^1_{\Gamma_D}(\Omega)$ and $\mathbf f\in[H^1_{\Gamma_D}(\Omega)]^2$ are given functions, such that $\int_{\partial\Omega\setminus\Gamma_D}\mathbf f\cdot\normal=0$.
 Then, there exists a constant $\alpha_\Omega\in(0,1]$ such that, for any $\alpha\in(0,\alpha_\Omega]$, the inequalities \ref{it:fo}-\ref{it:fvo} of \cref{lmm:actualRSlemma} hold for $w$ in the place of $\check w$.
\end{lemma}

These different boundary conditions play a crucial role in the proof of  \cref{lmm:RSpar3}.
Indeed, the regularity results holding for fully Dirichlet boundary conditions do not straightforwardly extend to the case of mixed conditions, for which some additional restrictions on the domain shape and regularity, and on the boundary data, have to be taken into account.
In the framework outlined at the beginning of the present section, we can resort to the regularity results of \cite{Grisvard,Lorenzi}.
In the following we report the proof of \cref{lmm:RSpar3}: we will follow the lines of that of \cref{lmm:actualRSlemma}, showing where the above-cited regularity results are employed and how the boundary integral terms - appearing in the case of mixed conditions - are dealt with.
For ease of notation, throughout the present section, $c$ will denote any positive constant that depends at most on the domain $\Omega$.
The value of this constant may vary from line to line and even within a single line.

\begin{proof}[Proof of \cref{lmm:RSpar3}]
 The proof builds on a bound for the complete $H^2(\Omega)$ norm of $w$ in terms of its Laplacian, in the form
 \begin{equation}\label{eq:Hlapl}
  \|w\|_{H^2(\Omega)} \leq c\left(\|\Delta w\|_{L^2(\Omega)} + \|w\|_{L^2(\Omega)}\right),
 \end{equation}
 that can be found, for a generic polygon, in \cite[Theorem 4.3.1.4]{Grisvard}.
 To simplify the notation, the dependence of $\sigma_{\mathbf z,\zeta}$ on $\mathbf z$ and $\zeta$ will be understood.

 Concerning point \ref{it:fvo}, the proof follows the lines of \cite{RannacherScott}, thanks to the fact that an inequality like \cref{eq:Hlapl}, involving $L^2$-type spaces, still holds if $L^{2/(2-\alpha)}$-type spaces are considered, for any $\alpha$ \cite[Theorem 4.3.2.4]{Grisvard}.

 Regarding point \ref{it:fo}, we follow the ideas of the proof of a similar result by \cite{RannacherScott}.
 To this aim, we need to collect the following two instrumental properties of the weight function $\sigma$.
 First, we notice that \cite[(2.2)]{GiraultNochettoScott}
 \begin{equation}\label{eq:GNS}
  |\nabla^k\sigma^\alpha| \leq C_{k,\alpha}\sigma^{\alpha-k},
 \end{equation}
 where the superscript $k$ denotes the $k$-th derivative order, and the constant $C_{k,\alpha}$ depends only on $k$ and $\alpha$.
 Moreover,
 \begin{equation}\label{eq:partnormsigmapos}
  \partial_\normal\sigma^\alpha = \alpha\sigma^{\alpha-1}\frac{(\mathbf x - \mathbf z)\cdot\normal}{\sigma} = \alpha\sigma^{\alpha-2}(\mathbf x - \mathbf z)\cdot\normal \quad \forall\mathbf x\in\partial\Omega,
 \end{equation}
 and being $\Omega$ convex, $\partial_\normal\sigma^\alpha\geq 0$ on the whole boundary $\partial\Omega$.
 
 Now we are ready to prove \ref{it:fo}. Since
 \begin{equation}
  \nabla^2(\sigma^{1+\alpha/2}w) = \sigma^{1+\alpha/2}\nabla^2w + w\nabla^2\sigma^{1+\alpha/2} + \nabla w\otimes\nabla\sigma^{1+\alpha/2} + \nabla\sigma^{1+\alpha/2}\otimes\nabla w,
 \end{equation}
 employing the triangle inequality, together with \cref{eq:GNS,eq:Hlapl}, yields
 \begin{align}
  \int_\Omega\sigma^{2+\alpha}&|\nabla^2w|^2 \leq \int_\Omega|\nabla^2(\sigma^{1+\alpha/2}w)|^2 + c\int_\Omega w^2\sigma^{\alpha-2} + c\int_\Omega|\nabla w|^2\sigma^\alpha\\
  &\leq c \int_\Omega\left(\sigma^{2+\alpha}|\nabla^2 w|^2 + w^2|\nabla^2\sigma^{1+\alpha/2}|^2 + 2|\nabla w|^2|\nabla\sigma^{1+\alpha/2}|^2\right)\\
    &\qquad + c\int_\Omega w^2\sigma^{\alpha-2} + c\int_\Omega|\nabla w|^2\sigma^\alpha\\
  &\leq c\int_\Omega\left(\sigma^{2+\alpha}|\Div\mathbf f|^2+w^2\sigma^{\alpha-2} + |\nabla w|^2\sigma^\alpha\right).
 \end{align}
 To control the last term at the right-hand side of this inequality, we observe that the weak formulation of problem \cref{eq:pbRSpar3} is 
 \begin{equation}\label{eq:weakpbRSpar3}\phantom{\ref{eq:weakpbRSpar3}}
  \int_\Omega\nabla w\cdot\nabla v = \int_\Omega v\,\Div\mathbf f \qquad \forall v\in H^1_{\Gamma_D}(\Omega).
 \end{equation}
 Therefore, recalling that $\partial_\normal\sigma^\alpha\geq0$ on $\partial\Omega$, the following steps can be performed:
 \begin{equation}\label{eq:sigmanablaw}\begin{aligned}
  \int_\Omega\sigma^\alpha|\nabla w|^2 &= \int_\Omega\nabla w\cdot\nabla(\sigma^\alpha w) - \frac{1}{2}\int_\Omega\nabla(w^2)\cdot\nabla\sigma^\alpha \\
  &\overset{(\ref{eq:weakpbRSpar3})}{=}\int_\Omega\Div(\mathbf f)\,\sigma^\alpha w + \frac{1}{2}\int_\Omega w^2\Delta\sigma^\alpha - \frac{1}{2}\int_{\partial\Omega}w^2\partial_\normal\sigma^\alpha\\
  &\overset{(\ref{eq:GNS})}{\leq} \int_\Omega\sigma^{\alpha+2}\Div\mathbf f\ \sigma^{-2}w + c\int_\Omega w^2\sigma^{\alpha-2}\\
  &\leq c\int_\Omega\sigma^{2+\alpha}|\Div\mathbf f|^2 + c\int_\Omega\sigma^{\alpha-2}w^2.
 \end{aligned}\end{equation}
 Now, to conclude the proof, a proper bound for $\int_\Omega\sigma^{\alpha-2}w^2$ is required.
 To this aim, we introduce a function $\phi$ solving the following problem:
 \begin{equation}\begin{cases}
  -\Delta\phi = \sgn(w)\,w^{2/\alpha} & \text{ in }\Omega,\\
  \phi = 0 & \text{ on }\Gamma_D,\\
  \partial_\normal\phi = 0 & \text{ on }\partial\Omega\setminus\Gamma_D.
 \end{cases}\end{equation}
 Being $w\in H^1(\Omega)$, it belongs to $L^s(\Omega)$ for any $s\in[1,\infty)$, in particular to $L^{1+\alpha/2}(\Omega)$.
 Thus the function $\phi$ belongs to $W^2_{1+\alpha/2}(\Omega)$, and an inequality similar to \cref{eq:Hlapl} holds for any $\alpha\neq0$ \cite[Theorem 4.3.2.4]{Grisvard}:
 \begin{equation}
  \|\phi\|_{W^2_{1+\alpha/2}(\Omega)} \leq c\left(\|\sgn(w)\,w^{2/\alpha}\|_{L^{1+\alpha/2}(\Omega)} + \|\phi\|_{L^{1+\alpha/2}(\Omega)}\right).
 \end{equation}
 This inequality, combined with the hypothesis $\int_{\partial\Omega\setminus\Gamma_D}\mathbf f\cdot\normal =0$ and a careful employment of H\"older inequality, yields
 \begin{align}
  \|w\|_{L^{1+2/\alpha}(\Omega)}^{1+2/\alpha} &= \int_\Omega w\,\sgn(w)\,w^{2/\alpha} = \int_\Omega \nabla w\cdot\nabla\phi = \int_\Omega\Div\mathbf f\ \phi \\
  &\leq \left|\int_\Omega\mathbf f\cdot\nabla\phi\right|\leq c\|\mathbf f\|_{L^{\frac{4+2\alpha}{2+3\alpha}}(\Omega)}\|w\|_{L^{1+2\alpha}(\Omega)}^{2/\alpha}.
 \end{align}
 whence
 \begin{equation}\label{eq:holderfancy}\begin{aligned}
  \|w\|_{1+2/\alpha}&\leq c\|\mathbf f\|_{L^{\frac{4+2\alpha}{2+3\alpha}}(\Omega)}= c\left(\int_\Omega\sigma^{(1+\alpha/2)\frac{4+2\alpha}{2+3\alpha}}|\mathbf f|^\frac{4+2\alpha}{2+3\alpha}\ \sigma^{-(1+\alpha/2)\frac{4+2\alpha}{2+3\alpha}}\right)^\frac{2+3\alpha}{4+2\alpha}\\
  &\leq\left(\int_\Omega\sigma^{2+\alpha}|\mathbf f|^2\right)^{1/2}\left(\int_\Omega\sigma^{-(2+\alpha)^2/(2\alpha)}\right)^{\alpha/(2+\alpha)},
 \end{aligned}\end{equation}
 where in the last step, H\"older inequality has been employed again.
 Now, noticing that (cf.~\cref{eq:GNS})
 \begin{equation}\label{eq:nablaksigmalambda}
  \|\nabla^k\sigma\|_{L^\infty(\Omega)} \leq c\zeta^{1-k},
 \end{equation}
 we can further bound \cref{eq:sigmanablaw} and \cref{eq:holderfancy} as
 \begin{align}
  \int_\Omega\sigma^\alpha|\nabla w|^2 &\leq c\int_\Omega\sigma^{2+\alpha}|\Div\mathbf f|^2 + c(\alpha^{-1}\zeta^{-\alpha})^{(2-\alpha)/(2+\alpha)}\|w\|_{1+2/\alpha}^2,\\
  \|w\|_{1+2/\alpha} &\leq c\zeta^{-(4+\alpha^2)/(4+2\alpha)}\left(\int_\Omega\sigma^{2+\alpha}|\mathbf f|^2\right)^{1/2}.
 \end{align}
 Merging these two inequalities gives thesis \ref{it:fo} for $\alpha_\Omega=1$.
\end{proof}

The inequalities of \cref{lmm:RSpar3} are instrumental to the proof of the following result, that actually states the stability of the Riesz projection operator defined in \cref{eq:Riesz}.

\begin{proposition}
\label{prop:projection}
 Let $\Gamma_D$ be a portion of a polygonal domain $\Omega\subset\mathbb R^2$, discretized as a regular mesh $\mathcal T_h$ having discretization step $h$.
 Then, the Riesz projection defined as in \cref{eq:Riesz} is stable in $W^1_p(\Omega)$, for any $p\in[1,\infty)$, i.e.\ \cref{eq:Rieszstable} holds independently of $p$.
\end{proposition}
\begin{proof}
 We follow the proof of a similar result, stated in \cite[section~2]{RannacherScott}, for the case of fully Dirichlet boundary conditions.
 The main difference lies in the boundary conditions imposed on the auxiliary problems that are going to be introduced.
 Anyway, thanks to \cref{lmm:RSpar3}, only little further difficulties will arise.
 For completeness, we report the whole proof in our framework.
Let us denote by $H^1_{\Gamma_D}(\Omega)$ the usual Hilbert space
 \begin{equation}
  H^1_{\Gamma_D}(\Omega) = \{\test\in W^1_2(\Omega) \st \test|_{\Gamma_D}=0\}.
 \end{equation}
Consider now a point $\mathbf z$ inside
 a triangle $K_z\in\mathcal T_h$ and let $\delta_z\in C_0^\infty(K_z)$ be an approximation of the Dirac delta concentrated in $\mathbf z$, such that \cite{Wihler,Scottdelta}
 \begin{gather}
  \int_\Omega\delta_z = 1,\qquad \|\nabla^k\delta_z\|_\infty\leq ch^{-2-k},\quad \forall k\in\mathbb N,\label{eq:deltazbound}\\
  \qquad\partial_i\varphi(\mathbf z)=\int_\Omega\delta_z\,\partial_i\varphi,\qquad \forall\varphi\in\Vo,\quad i=1,2.\label{eq:deltazisdelta}
 \end{gather}
 It is worthwhile to observe already at this early stage that the generic point $\mathbf z$ on which $\delta_z$ is concentrated belongs to the interior of $\Omega$: for the present proof, there will be no need to consider the possibility of choosing $\mathbf z$ on the boundary $\partial\Omega$.
 Fix $i\in\{1,2\}$ and let $g_z\in H^1_{{\Gamma_D}}(\Omega)$ be a regularized $i-$th derivative of the Green function for the Laplacian, defined as the solution of
 \begin{equation}\label{eq:pbgz}
  \int_\Omega\nabla g_z\cdot\nabla\varphi = \int_\Omega\delta_z\,\partial_i\varphi,\qquad \forall\varphi\in H^1_{{\Gamma_D}}(\Omega).
 \end{equation}
 Thence, combining \cref{eq:pbgz} with \cref{eq:deltazisdelta} and the definition \cref{eq:Riesz} of the Riesz operator yields
 \begin{align}
  \partial_iR_h\sol(\mathbf z) &\overset{(\ref{eq:deltazisdelta})}{=} \int_\Omega\delta_z\,\partial_iR_h\sol \overset{(\ref{eq:pbgz})}{=} \int_\Omega\nabla g_z\cdot\nabla R_h\sol \overset{(\ref{eq:Riesz})}{=} \int_\Omega\nabla R_hg_z\cdot\nabla R_h\sol\\
  &\overset{(\ref{eq:Riesz})}{=} \int_\Omega\nabla R_hg_z\cdot\nabla\sol = \int_\Omega\nabla g_z\cdot\nabla \sol - \int_\Omega\nabla\sol\cdot\nabla\left(g_z-R_hg_z\right)\\ &\overset{(\ref{eq:pbgz})}{=} \int_\Omega\delta_z\,\partial_i\sol - \int_\Omega\nabla\sol\cdot\nabla(g_z-R_hg_z).
 \end{align}
 Let us introduce the weight function
 \begin{equation}
  \sigma(x) = \sqrt{|\mathbf x-\mathbf z|+\kappa^2h^2},
 \end{equation}
 with a fixed $\kappa\geq1$ independent of $h$, for which, thanks to \cref{eq:nablaksigmalambda},
 \begin{equation}\label{eq:nablaksigma}
  \|\nabla^k\sigma\|_{L^\infty(\Omega)} \leq c(\kappa h)^{1-k}.
 \end{equation}
 Then, one can show that \cite[(2.6)]{RannacherScott}
 \begin{equation}
  \|\partial_iR_h\sol\|_{L^p(\Omega)}\leq c\|\nabla\sol\|_{L^p(\Omega)}\left(1+\frac{M_h}{\sqrt{\alpha h^\alpha}}\right),
 \end{equation}
 where $\alpha$ is a generic scalar in $(0,1]$ and
 \begin{equation}
  M_h = \max_{z\in\Omega}\|\sigma^{1+\alpha/2}\nabla(g_z-R_hg_z)\|_{L^2(\Omega)}.
 \end{equation}
 Therefore, a sufficient condition for the thesis of the present lemma is that $M_h\leq c_\alpha h^\alpha$ for a proper choice of $\kappa,\alpha$.
 The rest of the proof is, thus, devoted to show that the quantity $M_z=\|\sigma^{1+\alpha/2}\nabla(g_z-R_hg_z)\|_{L^2(\Omega)}$ fulfills $M_z\leq c_\alpha h^\alpha$, independently of $\mathbf z$.
 
 Introducing the quantity $\psi_z=\sigma^{2+\alpha}(g_z-R_hg_z)$ and employing the Galerkin orthogonality stemming from \cref{eq:Riesz}, we can rewrite
 \begin{align}
  M_z^2 &= \int_\Omega\sigma^{2+\alpha}|\nabla(g_z-R_hg_z)|^2 \\
  &= \int_\Omega\nabla(g_z-R_hg_z)\cdot\nabla(\psi_z-\interpol_h\psi_z) - \int_\Omega\nabla(g_z-R_hg_z)\cdot\nabla\sigma^{2+\alpha}(g_z-R_hg_z)\\
  &= \int_\Omega\nabla(g_z-R_hg_z)\cdot\nabla(\psi_z-\interpol_h\psi_z) - \frac{1}{2}\int_\Omega\nabla(g_z-R_hg_z)^2\cdot\nabla\sigma^{2+\alpha} \\
  &= \int_\Omega\nabla(g_z-R_hg_z)\cdot\nabla(\psi_z-\interpol_h\psi_z) + \frac{1}{2}\int_\Omega(g_z-R_hg_z)^2\Delta\sigma^{2+\alpha} \\
  &\qquad - \frac{1}{2}\int_{\partial\Omega}(g_z-R_hg_z)^2\partial_\normal\sigma^{2+\alpha},
 \end{align}
 where $\interpol_h$ denotes the classical Lagrange interpolator onto the piecewise linear finite element space $\Vo$.
 Being the domain $\Omega$ convex, the normal derivative of $\sigma^{2+\alpha}$ is positive (cf.\ \cref{eq:partnormsigmapos}), and hence,
 \begin{align}
  M_z^2 &\leq\int_\Omega\nabla(g_z-R_hg_z)\cdot\nabla(\psi_z-\interpol_h\psi_z) + \frac{1}{2}\int_\Omega(g_z-R_hg_z)^2\Delta\sigma^{2+\alpha} \\
  &\leq\frac{1}{2}\int_\Omega\sigma^{2+\alpha}|\nabla(g_z-R_hg_z)|^2 + \frac{1}{2}\int_\Omega\sigma^{-2-\alpha}|\nabla(\psi_z-\interpol_h\psi_z)|^2\\
  &\qquad + \frac{1}{2}\int_\Omega(g_z-R_hg_z)^2\Delta\sigma^{2+\alpha}\\
  &= \frac{1}{2}M_z^2 + \frac{1}{2}\int_\Omega\sigma^{-2-\alpha}|\nabla(\psi_z-\interpol_h\psi_z)|^2 + \frac{1}{2}\int_\Omega(g_z-R_hg_z)^2\Delta\sigma^{2+\alpha}.
 \end{align}
 Thanks to \cref{eq:GNS}, we can then obtain
 \begin{equation}
  M_z^2 \leq \int_\Omega\sigma^{-2-\alpha}|\nabla(\psi_z-\interpol_h\psi_z)|^2 + c\int_\Omega\sigma^\alpha(g_z-R_hg_z)^2.
 \end{equation}

Since $\max_{K\in\mathcal T_h}(\max_K\sigma\,/\,\min_K\sigma)\leq c$, the classical interpolation error estimate (see, e.g., \cite[Theorem 3.4.3]{QV}) can be extended to the weighted-norm case, namely
\begin{equation}\label{eq:interpol}
 \int_\Omega\sigma^\rho|\nabla(v-\interpol_hv)|^2 \leq ch^2\sum_{K\in\mathcal T_h}\int_K\sigma^\rho|\nabla^2v|^2
\end{equation}
holds for any $\rho\in\mathbb R$ and for any $v\in H^2(\Omega)$.
Thus, recalling the definition of $\psi_z$ and inequality \cref{eq:nablaksigma}, the interpolation error estimate \cref{eq:interpol} yields
\begin{align}\label{eq:prepreprelemma1}
 M_z^2 &\leq ch^2\sum_{K\in\mathcal T_h}\int_K\sigma^{-2-\alpha}\left[\sigma^{4+2\alpha}|\nabla^2(g_z-R_hg_z)|^2 +(g_z-R_hg_z)^2|\nabla^2\sigma^{2+\alpha}|^2 \right.\\
 &\qquad\qquad\qquad\qquad\qquad\left.+ 2|\nabla (g_z-R_hg_z)|^2|\nabla\sigma^{2+\alpha}|^2\right]\\
 &\qquad+ c\int_\Omega\sigma^\alpha(g_z-R_hg_z)^2.
\end{align}
Now, observing that $\nabla^2R_hg_z=0$, because $R_hg_z$ is piecewise linear, and employing \cref{eq:GNS,eq:nablaksigma} gives
\begin{align}\label{eq:preprelemma1}
 M_z^2&\leq ch^2\int_\Omega\sigma^{2+\alpha}|\nabla^2g_z|^2 + c\int_\Omega(g_z-R_hg_z)^2\left(\sigma^\alpha+ch^2\sigma^{-2-\alpha}\right) \\
 &\qquad + 2ch^2\int_\Omega\sigma^{-2}\sigma{\alpha+2}|\nabla(g_z-R_hg_z)|^2\\
 &\leq ch^2\int_\Omega\sigma^{2+\alpha}|\nabla^2g_z|^2 + c\kappa^{-2}\int_\Omega\sigma^{2+\alpha}|\nabla(g_z-R_hg_z)|^2 \\
 &\qquad + c(1+\kappa^{-2})\int_\Omega\sigma^\alpha(g_z-R_hg_z)^2,
\end{align}
whence, for $\kappa$ large enough,
\begin{equation}\label{eq:prelemma1}
 M_z^2\leq ch^2\int_\Omega\sigma^{2+\alpha}|\nabla^2g_z|^2 + c\int_\Omega\sigma^\alpha(g_z-R_hg_z)^2.
\end{equation}

In order to control the last term of \cref{eq:prelemma1}, we introduce the following auxiliary problem:
\begin{equation}
 \begin{cases}
  -\Delta w = \sigma^\alpha(g_z-R_hg_z) & \text{ in }\Omega,\\
  w = 0 & \text{ on }\Gamma_D,\\
  \partial_\normal w=0 & \text{ on }\partial\Omega\setminus\Gamma_D.
 \end{cases}
\end{equation}
Thanks to \cref{lmm:RSpar3} and \cref{eq:GNS}, the solution $w$ to the problem belongs to\linebreak[4] $H^1_{\Gamma_D}(\Omega)\cap H^2(\Omega)$, and the following inequality holds:
\begin{align}
 \int_\Omega\sigma^{-2-\alpha}|\nabla^2 w|^2 &\leq c\alpha^{-1}(\kappa h)^{-2}\int_\Omega\sigma^{2-\alpha}|\nabla[\sigma^\alpha(g_z-R_hg_z)]|^2\\
 &\leq c\alpha^{-1}(\kappa h)^{-2}\int_\Omega\left[\sigma^\alpha(g_z-R_hg_z)^2 + \sigma^{2+\alpha}|\nabla(g_z-R_hg_z)|^2\right]\\
 &= c\alpha^{-1}(\kappa h)^{-2}\left[M_z^2 + \int_\Omega \sigma^\alpha(g_z-R_hg_z)^2\right].
\end{align}

Being $(g_z-R_hg_z)\in H^1_{\Gamma_D}(\Omega)$, and resorting again to the $H^1$-orthogonality of this function w.r.t.\ the discrete space, the last integral of \cref{eq:prelemma1} can be bounded as follows:
\begin{equation}\label{eq:prepostlemma1}\begin{aligned}
 \int_\Omega\sigma^\alpha(g_z-R_hg_z)^2 &= \int_\Omega\nabla(w-\interpol_hw)\cdot\nabla(g_z-R_hg_z) \\
 &\leq M_z\left(\int_\Omega\sigma^{-2-\alpha}|\nabla(w-\interpol_hw)|^2\right)^{1/2}\\
 &\leq c(\alpha\kappa)^{-1}M_z^2 + c\,\alpha\kappa\,h^2\int_\Omega\sigma^{-2-\alpha}|\nabla^2w|^2\\
 &\leq c(\alpha\kappa)^{-1}M_z^2 + c\,\kappa^{-1}\left[M_z^2 + \int_\Omega\sigma(g_z-R_hg_z)^2\right] \\
 &\leq 2c(\alpha\kappa)^{-1} M_z^2 + c\kappa^{-1}\int_\Omega\sigma(g_z-R_hg_z)^2,
\end{aligned}\end{equation}
whence, for $\kappa$ large enough,
\begin{equation}\label{eq:postlemma1}
 \int_\Omega\sigma^\alpha(g_z-R_hg_z)^2 \leq \frac{c}{\kappa-1}M_z^2.
\end{equation}

 Then, combining \cref{eq:prelemma1} and \cref{eq:postlemma1} and choosing $\kappa$ large enough provides
 \begin{equation}
  M_z^2\leq ch^2\int_\Omega\sigma^{2+\alpha}|\nabla^2g_z|^2.
 \end{equation}

 In the last step of the proof, we employ \ref{it:fo} of \cref{lmm:RSpar3} on $g_z$, with $\mathbf f = \delta_z\,\mathbf e_i$.
 Indeed, since $\delta_z|_{\partial\Omega}=0$ and $g_z$ fulfills \cref{eq:pbgz}, $g_z$ is the solution of
 \begin{equation}\begin{cases}
  -\Delta g_z = \Div \mathbf f & \text{ in }\Omega, \\
  g_z = 0 & \text{ on }\Gamma_D, \\
  \partial_\normal g_z = 0 & \text{ on }\partial\Omega\setminus\Gamma_D.
 \end{cases}\end{equation}
 Thus, the following a priori estimate holds:
 \begin{equation}
  M_z^2 \leq c\,h^2\left[\int_\Omega\sigma^{2+\alpha}|\nabla\delta_z|^2 + \alpha^{-1}(\kappa h)^{-2}\int_\Omega\sigma^{2+\alpha}|\delta_z|^2\right],  
 \end{equation}
 whence, recalling also \cref{eq:deltazbound,eq:nablaksigma},
 \begin{equation}
  M_z^2 \leq ch^\alpha + c\alpha^{-1}\kappa^{-2}h^\alpha.
 \end{equation}

 Eventually, choosing $\kappa$ large enough, a bound of the form $M_z\leq c_\alpha h^\alpha$ is proven.
 Since the right-hand side of such inequality does not depend on the point $\mathbf z$, this concludes the proof.
\end{proof}

%%%%%%%%%%%%%%%%%%%%%%%%%%%%%%%%%%%%%%%%%%%%%%%%%%%%%%%%%%%%%%%%%%%%%%%%%%%%%%%%
%%%%%%%%%%%%%%%%%%%%%%%%%%%%%%%%%%%%%%%%%%%%%%%%%%%%%%%%%%%%%%%%%%%%%%%%%%%%%%%%
\section{Conclusions}
\label{sec:conclusioni}
\addcontentsline{toc}{section}{\nameref{sec:conclusioni}}
%%%%%%%%%%%%%%%%%%%%%%%%%%%%%%%%%%%%%%%%%%%%%%%%%%%%%%%%%%%%%%%%%%%%%%%%%%%%%%%%
%%%%%%%%%%%%%%%%%%%%%%%%%%%%%%%%%%%%%%%%%%%%%%%%%%%%%%%%%%%%%%%%%%%%%%%%%%%%%%%%

The present work has dealt with the theoretical and numerical analysis of a free boundary problem for the Laplacian with mixed boundary conditions,
where the contact points were free to move, and contact angles have been enforced.
The treatment of this latter condition is new, in this context.
Uniqueness and local existence of the solution of the continuous problem have been proved, via a fixed-point argument.
The proof has hinged upon the suitable definition of a lifting operator extending functions defined on the free surface.
Then, piecewise linear finite elements have been introduced to discretize both the free-boundary function $\omega$ and the bulk solution $\sol$.
In these settings, the Riesz projector onto the discrete bulk space has been proved to be stable with respect to the $W^1_p$ norm.
Finally, this result has been employed to prove the well-posedness and the optimal convergence of the discrete approximation.

\appendix

%%%%%%%%%%%%%%%%%%%%%%%%%%%%%%%%%%%%%%%%%%%%%%%%%%%%%%%%%%%%%%%%%%%%%%%%%%%%%%%%
%%%%%%%%%%%%%%%%%%%%%%%%%%%%%%%%%%%%%%%%%%%%%%%%%%%%%%%%%%%%%%%%%%%%%%%%%%%%%%%%
\section*{Acknowledgments}
%%%%%%%%%%%%%%%%%%%%%%%%%%%%%%%%%%%%%%%%%%%%%%%%%%%%%%%%%%%%%%%%%%%%%%%%%%%%%%%%
%%%%%%%%%%%%%%%%%%%%%%%%%%%%%%%%%%%%%%%%%%%%%%%%%%%%%%%%%%%%%%%%%%%%%%%%%%%%%%%%

The author thanks the Department of Mathematics of the University of Maryland, College Park, and in particular R.H.~Nochetto, for hosting him for a visiting period, during which part of the present work was carried out. Both R.H.~Nochetto and H.~Antil are thanked for the interesting interactions in the preliminary stage of the present work.
The author also shows his gratitude to N.~Parolini and M.~Verani, for the useful discussions about different aspects of the results presented here.
% Moxoff s.p.a.\ is gratefully acknowledged for the financial support to this research activity.


\begin{thebibliography}{10}

\bibitem{ANS14}
{\sc H.~Antil, R.~H. Nochetto, and P.~Sodr\'e}, {\em Optimal control of a free
  boundary problem: analysis with second-order sufficient conditions}, SIAM J.
  Control Optim., 52 (2014), pp.~2771--2799.

\bibitem{ASV88}
{\sc D.~N. Arnold, L.~R. Scott, and M.~Vogelius}, {\em Regular inversion of the
  divergence operator with {D}irichlet boundary conditions on a polygon}, Ann.
  Scuola Norm. Sup. Pisa Cl. Sci. (4), 15 (1988), pp.~169--192 (1989).

\bibitem{BaiChooChungKim}
{\sc K.~Bai, S.~Choo, S.~Chung, and D.~Kim}, {\em Numerical solutions for
  nonlinear free surface flows by finite element methods}, Appl. Math. Comput.,
  163 (2005), pp.~941 -- 959.

\bibitem{Eppler}
{\sc K.~Eppler, H.~Harbrecht, and R.~Schneider}, {\em On convergence in
  elliptic shape optimization}, SIAM J. Control Optim., 46 (2007), pp.~61--83.

\bibitem{ErnGuermond}
{\sc A.~Ern and J.-L. Guermond}, {\em Theory and practice of finite elements},
  vol.~159 of Applied Mathematical Sciences, Springer-Verlag, New York, 2004.

\bibitem{axisymm}
{\sc K.~M. Forward and G.~C. Rutledge}, {\em Free surface electrospinning from
  a wire electrode}, Chemical Engineering Journal, 183 (2012), pp.~492 -- 503.

\bibitem{MovingCL}
{\sc I.~Fumagalli, N.~Parolini, and M.~Verani}, {\em On a free-surface problem
  with moving contact line: from variational principles to stable numerical
  approximations}.
\newblock Under review (MOX preprint 03/2017,
  \url{http://mox.polimi.it/publication-results/?id=649&tipo=add_qmox}).

\bibitem{FPVShOpt}
{\sc I.~Fumagalli, N.~Parolini, and M.~Verani}, {\em Shape optimization for
  {S}tokes flows: a finite element convergence analysis}, ESAIM Math. Model.
  Numer. Anal., 49 (2015), pp.~921--951.

\bibitem{Gerbeau}
{\sc J.-F. Gerbeau and T.~Lelièvre}, {\em Generalized {Navier} boundary
  condition and geometric conservation law for surface tension}, Comput. Method
  Appl. M., 198 (2009), pp.~644 -- 656.

\bibitem{GiraultNochettoScott}
{\sc V.~Girault, R.~H. Nochetto, and R.~Scott}, {\em Maximum-norm stability of
  the finite element {S}tokes projection}, J. Math. Pures Appl. (9), 84 (2005),
  pp.~279--330.

\bibitem{Grisvard}
{\sc P.~Grisvard}, {\em Elliptic problems in nonsmooth domains}, vol.~24 of
  Monographs and Studies in Mathematics, Pitman (Advanced Publishing Program),
  Boston, MA, 1985.

\bibitem{Wihler}
{\sc {Houston, Paul} and {Wihler, Thomas Pascal}}, {\em Discontinuous
  {G}alerkin methods for problems with {D}irac delta source}, ESAIM Math.
  Model. Numer. Anal., 46 (2012), pp.~1467--1483.

\bibitem{KinigerVexler}
{\sc B.~Kiniger and B.~Vexler}, {\em A priori error estimates for finite
  element discretizations of a shape optimization problem}, ESAIM Math. Model.
  Numer. Anal., 47 (2013), pp.~1733--1763.

\bibitem{poincareMedia}
{\sc E.~H. Lieb and M.~Loss}, {\em Analysis}, vol.~14 of Graduate Studies in
  Mathematics, American Mathematical Society, Providence, RI, second~ed., 2001.

\bibitem{Lorenzi}
{\sc A.~Lorenzi}, {\em A mixed problem for the {L}aplace equation in a right
  angle with an oblique derivative given on a side of the angle}, Ann. Mat.
  Pur. Appl., 100 (1974), pp.~259--306.

\bibitem{Scardovelli}
{\sc S.~Manservisi and R.~Scardovelli}, {\em A variational approach to the
  contact angle dynamics of spreading droplets}, Comput. \& Fluids, 38 (2009),
  pp.~406--424.

\bibitem{Meyers}
{\sc N.~G. Meyers}, {\em An {$L^{p}$}-estimate for the gradient of solutions of
  second order elliptic divergence equations}, Ann. Scuola Norm. Sup. Pisa (3),
  17 (1963), pp.~189--206.

\bibitem{QV}
{\sc A.~Quarteroni and A.~Valli}, {\em Numerical approximation of partial
  differential equations}, vol.~23 of Springer Series in Computational
  Mathematics, Springer-Verlag, Berlin, 1994.

\bibitem{RannacherScott}
{\sc R.~Rannacher and R.~Scott}, {\em Some optimal error estimates for
  piecewise linear finite element approximations}, Math. Comp., 38 (1982),
  pp.~437--445.

\bibitem{SS91}
{\sc P.~Saavedra and L.~R. Scott}, {\em Variational formulation of a model
  free-boundary problem}, Math. Comp., 57 (1991), pp.~451--475.

\bibitem{Scottdelta}
{\sc R.~Scott}, {\em Finite element convergence for singular data}, Numer.
  Math., 21 (1973/74), pp.~317--327.

\bibitem{Walker}
{\sc S.~W. Walker}, {\em A mixed formulation of a sharp interface model of
  {Stokes} flow with moving contact lines}, ESAIM Math. Model. Numer. Anal., 48
  (2014), pp.~969--1009.

\bibitem{Yamamoto}
{\sc Y.~Yamamoto, T.~Ito, T.~Wakimoto, and K.~Katoh}, {\em Numerical
  simulations of spontaneous capillary rises with very low capillary numbers
  using a front-tracking method combined with generalized {Navier} boundary
  condition}, Int. J. Multiphase Flow, 51 (2013), pp.~22 -- 32.

\end{thebibliography}
\end{document}